\documentclass[10pt,reqno]{amsart}

\usepackage{amssymb,amsrefs,amsmath,mathtools}
\usepackage{a4wide,nicefrac}
\usepackage[shortlabels]{enumitem}
\usepackage{hyperref}
\usepackage{ourbib}
\usepackage{xcolor}

\hypersetup{
    colorlinks,
    linkcolor={red!50!black},
    citecolor={blue!50!black},
    urlcolor={blue!80!black}
}

\newcommand{\R}{\mathbb{R}}

\newcommand{\vol}{\mathsf{vol}}
\newcommand{\scal}{\mathrm{scal}}
\newcommand{\Mk}{\mathbb{M}_\kappa^{n+1}}
\newcommand{\<}{\langle}
\renewcommand{\>}{\rangle}
\renewcommand{\phi}{\varphi}
\renewcommand{\Re}{\mathsf{Re}}

\DeclareMathOperator{\id}{\mathrm{id}}
\DeclareMathOperator{\End}{\mathsf{End}}

\newtheorem{theorem}{Theorem}
\newtheorem*{conjecture}{Conjecture}
\newtheorem{lemma}{Lemma}
\newtheorem{proposition}{Proposition}

\theoremstyle{definition}
\newtheorem*{remark}{Remark}
\newtheorem{example}{Example} 

\allowdisplaybreaks{}

\begin{document}

\title{Upper bound for the total mean curvature of spin fill-ins} 
\author{Christian B\"ar}
\address{Universit\"at Potsdam, Institut f\"ur Mathematik, 14476 Potsdam, Germany}
\email{\href{mailto:christian.baer@uni-potsdam.de}{christian.baer@uni-potsdam.de}}
\urladdr{\url{https://www.math.uni-potsdam.de/baer/}}

\begin{abstract} 
Gromov conjectured that the total mean curvature of the boundary of a compact Riemannian manifold can be estimated from above by a constant depending only on the boundary metric and on a lower bound for the scalar curvature of the fill-in.
We prove Gromov's conjecture if the manifolds are spin with a constant that also depends on a lower bound on the mean curvature $H$ (which is allowed to take negative values).
If the boundary is a (not necessarily convex) hypersurface in a space form of non-negative curvature, then the constant can be made explicit in terms of the mean curvature of this model embedding.
If the boundary has constant sectional curvature $\kappa>0$ and is a projective space of dimension $n\equiv 3 \mod 4$ or a sphere, then the constant can be expressed in terms of $\kappa$.
If the boundary is a flat torus, then the constant can be expressed in terms of lattice data.
\end{abstract}



\date{\today}

\maketitle

\section{Introduction} 

Given a compact Riemannian spin manifold $M$ without boundary and a number $\lambda\ge0$, we may consider all compact connected manifolds $X$ that bound $M$ as a Riemannian spin manifold and satisfy the scalar curvature bound $\scal_X\ge -\lambda^2$.
Given such an $X$, let $H$ be the unnormalized mean curvature of the boundary $\partial X=M$.
Our sign convention for the mean curvature is such that the boundary of an $(n+1)$-dimensional Euclidean ball of radius $r$ has mean curvature $H=n/r$.

Of course, $M$ may not bound any compact manifold $X$, but if it does, it follows from \cite{SWW}*{Theorem~1.1} by Shi, Wang, and Wei that the given metric on $M$ can be extended to a metric on $X$ with $\scal_X>0$.
Their result gives no control over the boundary mean curvature.
Theorem~3.7 in \cite{BH} by the author and Hanke implies that the metric on $X$ can be deformed in such a way that its scalar curvature remains positive, it still induces the given metric on $M$, and the mean curvature of the boundary becomes arbitrarily negative.
Thus
\[
\inf \int_M H = -\infty
\]
where the infimum is taken over $X$ with $\scal_X>0$ and $\partial X=M$ (as Riemannian manifolds).

While we can make $H$ as small as we like, Gromov conjectured that $\int_M H$ cannot be made arbitrarily large.
More precisely, the following is believed to hold:
\begin{conjecture}[Gromov \cite{G}*{p.~232}]
Let $M$ be a compact Riemannian manifold without boundary which bounds a compact manifold and let $\sigma\in\R$.
Then there exists a constant $C(M,\sigma)$ such that for each compact Riemannian manifold $X$ with boundary $\partial X=M$ and scalar curvature $\scal_X\ge \sigma$ we have
\[
\int_M H \le C(M,\sigma).
\]
\end{conjecture}

If $\dim(M)=1$ and $\sigma=0$, then this follows from the Gauss-Bonnet theorem.
Indeed, assuming without loss of generality that $X$ is connected, the Gauss-Bonnet theorem gives us
\[
\int_M H = 2\pi\chi(X) - \frac12 \int_X\scal_X \le 2\pi.
\]
If we are more modest and replace the integral of the mean curvature by its minimum $H_{\min} = \min_M H$, then we indeed have
\begin{equation}
H_{\min} \le C(M,\sigma)
\label{eq.HminEstimate}
\end{equation}
for spin manifolds.
This follows from \cite{HMR}*{Theorem~2} by Hijazi, Montiel, and Rold\'an.
Here the constant $C(M,\sigma)$ depends on the first non-negative eigenvalue of the Dirac operator of $M$ and on $\sigma$.
Combining this with an upper Dirac eigenvalue estimate by the author (\cite{Baer}*{Main Theorem}), one obtains \eqref{eq.HminEstimate} where the constant depends on the hyperspherical radius of $M$ and on $\sigma$, as has been observed by Brendle, Tsiamis, and Wang in \cite{BTW}.
See also \cite{CHZ}*{Theorem~1.5} by Cecchini, Hirsch, and Zeidler for \eqref{eq.HminEstimate} in the case $\sigma=0$.
Clearly, estimate \eqref{eq.HminEstimate} is only meaningful if the mean curvature is positive.

We prove Gromov's conjecture for spin manifolds with a constant that also depends on a lower bound for the mean curvature.
More precisely, we show the following:

\begin{theorem}
\label{thm.main}
Let $X$ be a compact Riemannian spin manifold of dimension $n+1\ge2$ with smooth boundary $\partial X=M$.
Let $\lambda\ge0$ be such that the scalar curvature of $X$ satisfies $\scal_X\ge -\lambda^2$.
Let $\eta\ge0$ be such that the mean curvature $H$ of the boundary satisfies $H\ge-\eta$.
Then we have
\[
\frac{1}{\vol(M)}\int_M H
\le 
C(M) + \sqrt{\tfrac{n}{n+1}}\,\lambda + \eta,
\]
where $C(M)$ is a constant depending only on the Riemannian spin manifold $M$.
\end{theorem}
In Example~\ref{ex.balls}, we choose $X$ to be a ball in a space form and see that the term $\sqrt{\tfrac{n}{n+1}}\,\lambda$ gives the precise asymptotics for $\lambda\to\infty$.
Thus, the coefficient $\sqrt{\tfrac{n}{n+1}}$ is optimal.

Most previously known results require the assumption that $H\ge0$ or even $H>0$.
The case where $M$ is diffeomorphic to a sphere with $\lambda=0$, $H>0$, and $X$ spin has been treated by Shi, Wang, and Wei in \cite{SWW}*{Theorem~4.1}.
Mantoulidis and Miao investigated in \cite{MM} the case where $n=2$ and $M$ has arbitrary topology, again if $\lambda=0$ and $H>0$.
In these results, the constant $C(M,\lambda)$ is not explicit.

Using entirely different surgery and deformation methods, Frenck, Hanke, and Hirsch have presented an independent proof of Gromov's conjecture with lower mean curvature bound in \cite{FHH}.
Their estimate is of the form
$$
\frac{1}{\vol(M)}\int_M H
\le 
C(M)(1 + \lambda + \eta)
$$
which also gives linear growth in $\lambda$ and $\eta$, but with unknown coefficients.
Their approach has the advantage that it allows one to drop the spin assumption if $n\le6$ and $\eta=0$.
The dimension restriction is due to the use of the hyperbolic positive mass theorem.

If we make additional assumptions on $M$, the constant $C(M)$ can be made explicit.
For $\kappa\in\R$, let $\Mk$ be the $(n+1)$-dimensional simply connected space form of constant sectional curvature $\kappa$.
For $\kappa=0$, this is the Euclidean space $\R^{n+1}$, for $\kappa=1$ the round unit sphere $S^{n+1}$, and for $\kappa=-1$ the hyperbolic space $\mathbb{H}^{n+1}$.

\begin{theorem}
\label{thm.hypersurface}
In addition to the assumptions in Theorem~\ref{thm.main}, assume that $M$ is isometric to an oriented hypersurface of $\Mk$ where $\kappa\ge0$.
Denote the (unnormalized) mean curvature of this embedding by $H_0$.
Then we have:
\begin{equation*}
\frac{1}{\vol(M)}\int_M H
\le 
2\,\sqrt{\frac{1}{\vol(M)}\int_M H_0^2} +  2n\sqrt{\kappa} + \sqrt{\tfrac{n}{n+1}}\,\lambda + \eta.
\end{equation*}
\end{theorem}

\begin{example}
The case $\kappa=\lambda=0$ and $H>0$ has been studied by Shi and Tam in \cite{ST}*{Theorem~4}.
Our estimate states that
\begin{equation}
\frac{1}{\vol(M)}\int_M H
\le 
2\,\sqrt{\frac{1}{\vol(M)}\int_M H_0^2}.
\label{eq.ShiTamNonconvex}
\end{equation}
while Shi and Tam proved the sharper estimate
\[
\int_M H
\le 
\int_M H_0
\]
under the additional assumption that each connected component of $M$ is a \emph{strictly convex} hypersurface in $\R^{n+1}$.
This strict convexity has been relaxed by Eichmair, Miao, and Wang in \cite{EMW} to require the components to have non-negative scalar curvature and admit embeddings as star-shaped hypersurfaces.

We do not need any convexity assumption on $M$.
In particular, the components of $M$ need not be diffeomorphic to a sphere.
Moreover, we only need $H\ge0$ rather than $H>0$ to obtain \eqref{eq.ShiTamNonconvex}.
\end{example}

\begin{theorem}
\label{thm.eigenvalue}
In addition to the assumptions in Theorem~\ref{thm.main}, assume that $M$ is isometric to one of the following:
\begin{itemize}
\item a flat torus,
\item a round sphere,
\item round projective space $\R P^n$ with $n\equiv 3 \mod 4$.
\end{itemize}
Denote the eigenvalue of the Dirac operator of $M$ of smallest modulus by $\mu_1$.
Then we have:
\begin{equation*}
\frac{1}{\vol(M)}\int_M H
\le 
4\,|\mu_1| + \sqrt{\tfrac{n}{n+1}}\,\lambda + \eta.
\end{equation*}
\end{theorem}

The Dirac eigenvalue $\mu_1$ is explicitly known in the cases mentioned above.
If $M$ is a round sphere or a round projective space of sectional curvature $\kappa$, then $|\mu_1| = \frac{n}{2}\sqrt{\kappa}$.
If $M$ is a flat torus, then $|\mu_1|$ can be expressed in terms of lattice data; see \cite{Fr}*{Satz~2} for details.
The $n$-dimensional torus has $2^n$ spin structures, and for one of them we have that $\mu_1=0$.
This is precisely the spin structure for which the torus does not bound as a spin manifold.

Wang has proved an estimate for flat tori in terms of a spin systole in \cite{W}*{Theorem~1.6} under the assumption that $\lambda=\sqrt{n(n+1)}$, $H>0$, $2\le n\le6$, and $X$ is diffeomorphic to a $2$-disk times a torus.

The paper is structured as follows:
In Section~\ref{sec.BVP} we prove the well-posedness of a boundary value problem for the Dirac operator.
This may be of independent interest.
Section~\ref{sec.prep} contains some technical preparation.
The next three sections contain the proofs of the theorems.
The main argument is contained in Section~\ref{sec.proof}.
The proofs of Theorems~\ref{thm.hypersurface} and \ref{thm.eigenvalue} improve the constant $C(M)$ in the special situation under consideration and make it explicit.
The final section contains the example of balls in a space form and some concluding remarks.

\medskip

\emph{Acknowledgments:} 
This work was supported by the Deutsche Forschungsgemeinschaft (DFG, German Research Foundation) -- project ID 569831821.

\section{A boundary value problem}
\label{sec.BVP}

Let $X$ be a compact connected Riemannian spin manifold of dimension $n+1\ge2$ with boundary $\partial X=M$. 
Let $M$ carry the metric and the spin structure induced from $X$.
The spinor bundle $\Sigma X$ of $X$ is a Hermitian vector bundle with metric connection.
We choose the convention that the fiberwise scalar product of $\Sigma X$ is antilinear in the first argument and linear in the second argument.
Let $\gamma\colon TX \to \End(\Sigma X)$ denote the Clifford multiplication.
It satisfies the Clifford relations
\[
\gamma(v)\gamma(w) + \gamma(w)\gamma(v) = -\< v,w\> \id_{\Sigma X}
\]
for all tangent vectors $v,w \in T_pX$ and $p\in X$.
The Dirac operator of $X$ is denoted by $D$.

Denote the inward unit normal vector field along the boundary by $\nu$.
Then 
\[
s := i\gamma(\nu) \in \End(\Sigma X)
\]
is a self-adjoint involution.
It anticommutes with Clifford multiplication by tangent vectors of $M$.
Hence, the eigenspaces corresponding to the eigenvalues $\pm 1$ of $s$ have the same dimension.
We obtain an orthogonal vector bundle decomposition
\begin{equation}
\Sigma X|_M = \Sigma^+ M \oplus \Sigma^-M,
\label{eq.splitting}
\end{equation}
where $s$ acts on $\Sigma^\pm M$ as $\pm 1$.
We denote the orthogonal projections onto these subbundles by
\[
P^\pm \colon \Sigma X|_M \to \Sigma^\pm M. 
\]

If $n$ is even, then $\Sigma X|_M$ can be canonically identified with the spinor bundle $\Sigma M$ of $M$.
If $n$ is odd, then $\Sigma X|_M$ can be canonically identified with $\Sigma M\oplus \Sigma M$ (which is not the same decomposition as the splitting in \eqref{eq.splitting}).
Denoting the intrinsic Dirac operator of $M$ by $\tilde{D}_M$ and using these identifications, we set $D_M := \tilde{D}_M$ if $n$ is even and $D_M = \begin{pmatrix}
\tilde{D}_M & 0 \\ 0 & -\tilde{D}_M\end{pmatrix}$ if $n$ is odd.
Then $D_M$ anticommutes with $s$ in both cases (see \cite{Baer3}*{Proposition~2.3}) and hence interchanges the subbundles $\Sigma^+M$ and $\Sigma^-M$.
Moreover, by \cite{Baer3}*{Proposition~2.2}, we have the relation
\begin{equation}
\nabla_\nu
=
-\gamma(\nu)D - D_M + \tfrac12 H .
\label{eq.DiracRel}
\end{equation}
This implies, in particular, that $D_M$ is an adapted boundary operator for $D$ in the sense of \cite{BB}.

We denote by $L^2(X,E)$ the space of square-integrable sections of a vector bundle $E$ over $X$ and by $H^\alpha(X,E)$ the $L^2$-Sobolev space of order $\alpha$, and similarly for $M$.

\begin{lemma}
\label{lem.BVP1}
Let $X$ be a compact Riemannian spin manifold with boundary $\partial X=M$.
Let $\Phi\in H^1(X,\Sigma X)$ satisfy $(D+i\lambda)\Phi=0$ where $\lambda\ge0$.
Then we have
\[
\| P^-(\Phi|_M)\|_{L^2(M)} \ge \| P^+(\Phi|_M)\|_{L^2(M)}.
\]
Similarly, if $\lambda\le0$, then 
\[
\| P^-(\Phi|_M)\|_{L^2(M)} \le \| P^+(\Phi|_M)\|_{L^2(M)}.
\]
\end{lemma}

\begin{proof}
By Green's formula for the Dirac operator (see e.g.~\cite{BB}*{Lemma~2.6}) we have
\begin{align*}
0 
&=
\int_X \< (D+i\lambda)\Phi, \Phi\>  - \int_X \< \Phi, (D+i\lambda)\Phi\>  \\
&=
\int_X \< \Phi, (D-i\lambda)\Phi\>  - \int_M \< \gamma(\nu)\Phi, \Phi\> - \int_X \< \Phi, (D+i\lambda)\Phi\>  \\
&= 
- \int_M \< \gamma(\nu)\Phi, \Phi\> - 2i\lambda \|\Phi\|_{L^2(X)}^2 \\
&= 
-i \int_M \< s\,\Phi, \Phi\>  - 2i\lambda \|\Phi\|_{L^2(X)}^2 \\
&=
-i \int_M  \< P^+\Phi - P^-\Phi, P^+\Phi + P^-\Phi\>   - 2i\lambda \|\Phi\|_{L^2(X)}^2 \\
&=
-i\Big(\|P^+\Phi\|_{L^2(M)}^2 - \|P^-\Phi\|_{L^2(M)}^2 + 2\lambda \|\Phi\|_{L^2(X)}^2\Big).
\end{align*}
This implies
\begin{equation}
\|P^-\Phi\|_{L^2(M)}^2 
= 
\|P^+\Phi\|_{L^2(M)}^2  + 2\lambda \|\Phi\|_{L^2(X)}^2 ,
\label{eq.P+P-}
\end{equation}
which proves the lemma.
\end{proof}

As a consequence we find the well-posedness of the following boundary value problem:

\begin{proposition}
\label{prop.BVP2}
Let $X$ be a compact connected Riemannian spin manifold with boundary $\partial X=M$ and let $\lambda\ge0$.
For each $\phi\in H^{\nicefrac12}(M,\Sigma^+M)$ and $\Psi\in L^2(X, \Sigma X)$ there exists a unique $\Phi\in H^1(X,\Sigma X)$ such that
\[
(D-i\lambda)\Phi=\Psi \quad \text{and} \quad P^+(\Phi|_M) = \phi.
\]
\end{proposition}

\begin{proof}
Both boundary conditions $P^+(\Phi|_M)=0$ and $P^-(\Phi|_M)=0$ are $\infty$-regular elliptic in the sense of \cite{BB}.
Therefore the four operators
\begin{equation}
(D\pm i\lambda) \oplus P^\pm(\cdot|_M)\colon H^1(X,\Sigma X) \to L^2(X,\Sigma X) \oplus H^{\nicefrac12}(M,\Sigma^\pm M)
\label{eq.BVPinhomog}
\end{equation}
are Fredholm.
If $\Phi\in H^1(X,\Sigma X)$ lies in the kernel of $(D - i\lambda) \oplus P^+(\cdot|_M)$, then $(D - i\lambda)\Phi=0$ and $P^+(\Phi|_M)=0$.
By Lemma~\ref{lem.BVP1}, we then also have $P^-(\Phi|_M)=0$.
Hence, $\Phi|_M=0$.
By the unique continuation property for Dirac-type operators, we conclude $\Phi=0$ (see e.g.~the proof of \cite{BBHW}*{Corollary~B.2})\footnote{This also follows from \eqref{eq.P+P-} unless $\lambda=0$.}.
Thus, the operator 
\begin{equation}
(D - i\lambda) \oplus P^+(\cdot|_M)\colon H^1(X,\Sigma X) \to L^2(X,\Sigma X) \oplus H^{\nicefrac12}(M,\Sigma^+ M)
\label{eq.BVPinhomog+}
\end{equation}
is injective.
The adjoint boundary value problem is given by
\begin{equation*}
(D + i\lambda) \oplus P^-(\cdot|_M)\colon H^1(X,\Sigma X) \to L^2(X,\Sigma X) \oplus H^{\nicefrac12}(M,\Sigma^- M)
\end{equation*}
which, by the same reasoning, is also injective.
Therefore, the operator in \eqref{eq.BVPinhomog+} is surjective, and hence an isomorphism.
\end{proof}

\section{Some preparation}
\label{sec.prep}

For the proof of the theorem we need some technical preparation.
Only Proposition~\ref{prop.sections} will be needed later on.

\begin{lemma}
\label{lem.ONB}
Let $V$ be a finite-dimensional Euclidean or unitary vector space.
Let $e_1,...,e_n$ be an orthonormal basis of $V$.
Let $U\subset V$ be a $k$-dimensional subspace and let $P\colon V\to U$ be the orthogonal projection onto $U$.
Then we have:
\[
\sum_{j=1}^n |Pe_j|^2 = k.
\]
\end{lemma}

\begin{proof}
We choose an orthonormal basis $u_1,...,u_k$ of $U$.
Then we have 
\[
u_m = \sum_{j=1}^{m} \<u_m,e_j\>e_j
\quad
\text{ and hence }
\quad
|u_m|^2 = \sum_{j=1}^{k} |\<u_m,e_j\>|^2.
\]
Similarly, we have 
\[
|Pe_j|^2 = \sum_{m=1}^k |\<e_j,u_m\>|^2.
\]
We compute
\begin{align*}
\sum_{j=1}^n |Pe_j|^2
&=
\sum_{j=1}^n \sum_{m=1}^k |\<u_m,e_j\>|^2 
=
\sum_{m=1}^k |u_m|^2 
= 
k.
\qedhere
\end{align*}
\end{proof}

\begin{lemma}
\label{lem.Partition1}
Let $M$ be a manifold (possibly with boundary) and let $U_j\subset M$ be an open cover of $M$.
Then there exist smooth non-negative functions $\chi_j\in C^\infty(M,\R)$ such that $\mathrm{supp}(\chi_j)\subset U_j$ for every $j$ and $\sum_j \chi_j^2 \equiv 1$ on $M$.
\end{lemma}

\begin{proof}
It is well known that there is a smooth partition of unity $\psi_j\in C^\infty(M,\R)$ subordinate to the open cover, i.e., $\psi_j\ge 0$ and $\mathrm{supp}(\psi_j)\subset U_j$ for every $j$ and $\sum_j \psi_j \equiv 1$ on $M$.
Then
\[
u 
:= 
\sum_j\psi_j^2
\]
is smooth and positive on $M$.
The functions $\chi_j := \psi_j/\sqrt{u}$ have the desired properties.
\end{proof}

\begin{proposition}
\label{prop.sections}
Let $M$ be a compact manifold (possibly with boundary) and let $E\to M$ be a Hermitian or Riemannian vector bundle.
Then there exist finitely many sections $\phi_1,...,\phi_m \in C^\infty(M,E)$ such that for each subbundle $F\subset E$ of rank $k$ we have 
\[
\sum_{j=1}^m |P\phi_j|^2 \equiv k
\]
where $P\colon E\to F$ is the orthogonal projection onto $F$.
\end{proposition}

\begin{proof}
We choose a finite open cover $U_j$ of $M$ such that $E|_{U_j}$ is trivial for each $j=1,...,N$.
We further choose locally defined smooth sections $e_{j,1},...,e_{j,r}$ of $E|_{U_j}$ which form an orthonormal basis at each point of $U_j$.
By Lemma~\ref{lem.Partition1} there exist smooth functions $\chi_j\in C^\infty(M,\R)$ such that $\mathrm{supp}(\chi_j)\subset U_j$ for every $j$ and $\sum_j \chi_j^2 \equiv 1$ on $M$.
The sections $\chi_j\cdot e_{j,l}$ can be extended by zero to smooth sections in $C^\infty(M,E)$.
We define the sections $\phi_j$ by 
\[
\begin{aligned}
\phi_1 &:= \chi_1\cdot e_{1,1}, &\ldots, &\quad\phi_r := \chi_1\cdot e_{1,r},\\
&\quad\vdots && \\
\phi_{r(N-1)+1} &:= \chi_N\cdot e_{N,1}, &\ldots, &\quad\phi_{rN} := \chi_N\cdot e_{N,r}.
\end{aligned}
\]
We compute, using Lemmas~\ref{lem.ONB} and \ref{lem.Partition1}:
\begin{align*}
\sum_{j=1}^{rN} |P\phi_j|^2
&=
\sum_{j=1}^N \sum_{l=1}^r |P(\chi_j e_{j,l})|^2 
=
\sum_{j=1}^N \chi_j^2 \sum_{l=1}^r |P e_{j,l}|^2 
=
\sum_{j=1}^N \chi_j^2 \cdot k
=
k.
\qedhere
\end{align*}
\end{proof}

\section{Proof of Theorem~\ref{thm.main}} 
\label{sec.proof}

In this section we carry out the proof of Theorem~\ref{thm.main}.

\subsection{Definition of the constant \emph{C}(\emph{M})}
\label{sec.CM}

Let $M$ be an $n$-dimensional closed Riemannian spin manifold.
We apply Proposition~\ref{prop.sections} to the spinor bundle $E=\Sigma M$ of $M$ if $n$ is even and to $E=\Sigma M \oplus \Sigma M$ if $n$ is odd.
We rescale the resulting smooth sections $\phi_1,...,\phi_m \in C^\infty(M,E)$ such that 
\begin{equation}
\sum_{j=1}^m |P\phi_j|^2 \equiv 1
\label{eq.Pphi}
\end{equation}
whenever $P\colon E\to F$ is the orthogonal projection onto a subbundle $F\subset E$ of half the rank of $E$.

Let $\tilde{D}_M$ be the intrinsic Dirac operator of $M$ and set $D_M := \tilde{D}_M$ if $n$ is even and $D_M = \begin{pmatrix}\tilde{D}_M & 0 \\ 0 & -\tilde{D}_M\end{pmatrix}$ if $n$ is odd.
We define
\begin{equation}
C(M) := \frac{4}{\vol(M)}\cdot\sum_{j=1}^m \|D_M\phi_j\|_{L^2(M)} \cdot \|\phi_j\|_{L^2(M)} . 
\label{eq.CM}
\end{equation}

\subsection{The estimate}
Now let $X$ be a spin fill-in of $M$, i.e., $M=\partial X$ as a Riemannian spin manifold.
Without loss of generality, we may assume that $X$ is connected, as otherwise we can apply the following argument to each connected component separately.
Connected components of $X$ with empty boundary do not contribute to the integral of the mean curvature and can thus be ignored.

We identify $E=\Sigma X|_M$ as discussed in Section~\ref{sec.BVP}.
Recall the splitting \(E = \Sigma^+ M \oplus \Sigma^-M\) in \eqref{eq.splitting} and the orthogonal projections onto these subbundles \(P^\pm \colon E \to \Sigma^\pm M\).

Let $\lambda\ge0$ be such that $\scal_X\ge -\lambda^2$ on $X$.
We apply Proposition~\ref{prop.BVP2} to $\phi=P^+\phi_j$ and $\Psi=0$ to obtain spinors $\Phi_j \in H^1(X,\Sigma X)$ with $(D-i\sqrt{\tfrac{n+1}{n}}\tfrac{\lambda}{2})\Phi_j=0$ and  $P^+(\Phi_j|_M) = P^+\phi_j$.
Since the boundary condition $P^+(\cdot|_M)$ is $\infty$-regular, these spinors are smooth, $\Phi_j\in C^\infty(X,\Sigma X)$.
The Weitzenböck formula and integration by parts give:
\begin{align*}
0
&=
\int_X \Big\<\Big(D+i\sqrt{\tfrac{n+1}{n}}\tfrac{\lambda}{2}\Big)\Big(D-i\sqrt{\tfrac{n+1}{n}}\tfrac{\lambda}{2}\Big)\Phi_j,\Phi_j\Big\> \\
&=
\int_X \<(D^2 +\tfrac{n+1}{n}\tfrac{\lambda^2}{4})\Phi_j,\Phi_j\> \\
&=
\int_X \<(\nabla^*\nabla + \tfrac{1}{4}\scal_X +\tfrac{\lambda^2}{4}+\tfrac{1}{n}\tfrac{\lambda^2}{4})\Phi_j,\Phi_j\> \\
&\ge
\int_X \<(\nabla^*\nabla + \tfrac{1}{n}\tfrac{\lambda^2}{4})\Phi_j,\Phi_j\> .
\end{align*}
We introduce a new connection $\tilde{\nabla}$ on $\Sigma X$ by setting
\[
\tilde{\nabla}_X\Phi := \nabla_X\Phi + \tfrac{i}{\sqrt{n(n+1)}}\tfrac{\lambda}{2}\gamma(X)\Phi.
\]
One easily computes that
\[
\tilde{\nabla}^*\tilde{\nabla} = \nabla^*\nabla + \tfrac{1}{n}\tfrac{\lambda^2}{4} .
\]
It follows that
\begin{align*}
0
&\ge
\int_X \<\tilde{\nabla}^*\tilde{\nabla}\Phi_j,\Phi_j\>\\
&=
\int_X |\tilde{\nabla}\Phi_j|^2 + \int_M \<\tilde{\nabla}_\nu\Phi_j,\Phi_j\> \\
&\ge
\int_M \<\tilde{\nabla}_\nu\Phi_j,\Phi_j\>\\
&=
\int_M \<\nabla_\nu\Phi_j,\Phi_j\> + \frac{1}{\sqrt{n(n+1)}}\frac{\lambda}{2} \int_M \<s\Phi_j,\Phi_j\> .
\end{align*}
Using \eqref{eq.DiracRel}, we compute along the boundary:
\begin{align*}
\nabla_\nu\Phi_j
&=
-\gamma(\nu)D\Phi_j - D_M\Phi_j + \tfrac12 H\Phi_j \\
&=
i\, s\, D\Phi_j - D_M\Phi_j + \tfrac12 H\Phi_j \\
&=
-\sqrt{\tfrac{n+1}{n}}\tfrac{\lambda}{2}\, s\, \Phi_j - D_M\Phi_j + \tfrac12 H\Phi_j.
\end{align*}
Hence, we find
\begin{align*}
0
&\ge
\int_M \Big(-\sqrt{\tfrac{n+1}{n}}\tfrac{\lambda}{2} \<s\Phi_j,\Phi_j\> - \<D_M\Phi_j,\Phi_j\> + \tfrac12 H\, |\Phi_j|^2 + \tfrac{1}{\sqrt{n(n+1)}}\tfrac{\lambda}{2}\<s\Phi_j,\Phi_j\>\Big) \\
&=
\int_M \Big(\sqrt{\tfrac{n}{n+1}}\tfrac{\lambda}{2} \<P^-\Phi_j-P^+\Phi_j,P^+\Phi_j+P^-\Phi_j\> - \<D_M\Phi_j,\Phi_j\> + \tfrac12 H\, |\Phi_j|^2 \Big) \\
&=
\int_M \Big(\sqrt{\tfrac{n}{n+1}}\tfrac{\lambda}{2} (|P^-\Phi_j|^2-|P^+\Phi_j|^2) - \<D_M\Phi_j,\Phi_j\> + \tfrac12 H\, |\Phi_j|^2 \Big) \\
&\ge
\int_M \Big(-\sqrt{\tfrac{n}{n+1}}\tfrac{\lambda}{2} |P^+\Phi_j|^2 - \<D_M\Phi_j,\Phi_j\> + \tfrac12 H\, |\Phi_j|^2 \Big) .
\end{align*}
Summation over $j=1,\ldots,m$ gives
\begin{align}
\frac12 \int_M H  
&=
\frac12 \sum_{j=1}^m\int_M H |P^+\phi_j|^2  \notag\\
&=
\frac12 \sum_{j=1}^m\int_M H |P^+\Phi_j|^2  \notag\\
&=
\frac12 \sum_{j=1}^m\int_M \Big[H |\Phi_j|^2 - H|P^-\Phi_j|^2 \Big] \notag\\\
&\le 
\sum_{j=1}^m \Big[\int_M \<D_M\,\Phi_j,\Phi_j\> + \sqrt{\frac{n}{n+1}}\frac{\lambda}{2} \int_M |P^+\Phi_j|^2- \frac12\int_MH|P^-\Phi_j|^2  \Big] .
\label{eq.intH1}
\end{align}
For the first term we find
\begin{align}
\sum_{j=1}^m \int_M \<D_M\,\Phi_j,\Phi_j\> 
&=
\sum_{j=1}^m \int_M (\<D_M\,P^+\Phi_j,P^-\Phi_j\> + \<D_M\,P^-\Phi_j,P^+\Phi_j\> )\notag\\
&=
\sum_{j=1}^m \int_M (\<D_M\,P^+\Phi_j,P^-\Phi_j\> + \<P^-\Phi_j,D_M\,P^+\Phi_j\> )\notag\\
&=
2\Re \sum_{j=1}^m \int_M \<D_M\,P^+\Phi_j,P^-\Phi_j\> \notag\\
&=
2\Re \sum_{j=1}^m \int_M \<D_M\,P^+\phi_j,P^-\Phi_j\> \notag\\
&=
2\Re \sum_{j=1}^m \int_M \<D_M\,\phi_j,P^-\Phi_j\> \notag\\
&\le 
2\sum_{j=1}^m \|D_M\phi_j\|_{L^2(M)} \cdot \|P^-\Phi_j\|_{L^2(M)}  \notag\\
&\le
2\sum_{j=1}^m \|D_M\phi_j\|_{L^2(M)} \cdot \|P^+\Phi_j\|_{L^2(M)}  \label{eq.pm}\\
&=
2\sum_{j=1}^m \|D_M\phi_j\|_{L^2(M)} \cdot \|P^+\phi_j\|_{L^2(M)}  \notag\\
&\le
2\sum_{j=1}^m \|D_M\phi_j\|_{L^2(M)} \cdot \|\phi_j\|_{L^2(M)}  \notag\\
&=
\tfrac12 C(M)\vol(M), 
\label{eq.intH2}
\end{align}
where we applied Lemma~\ref{lem.BVP1} in \eqref{eq.pm}.
For the second term \eqref{eq.intH1} we get
\begin{align}
\sum_{j=1}^m \sqrt{\frac{n}{n+1}}\frac{\lambda}{2} \int_M |P^+\Phi_j|^2 
&=
\sqrt{\frac{n}{n+1}}\frac{\lambda}{2} \sum_{j=1}^m \int_M |P^+\phi_j|^2
=
\sqrt{\frac{n}{n+1}}\frac{\lambda}{2} \vol(M).
\label{eq.intH3}
\end{align}
For the third term in \eqref{eq.intH1} we obtain, again using Lemma~\ref{lem.BVP1}:
\begin{align}
-\frac12 \sum_{j=1}^m \int_M H |P^-\Phi_j|^2
&\le
\frac{\eta}{2}\sum_{j=1}^m  \|P^-\Phi_j\|^2
\le
\frac{\eta}{2}\sum_{j=1}^m  \|P^+\Phi_j\|^2
=
\frac{\eta}{2}\sum_{j=1}^m  \|P^+\phi_j\|^2
=
\frac{\eta}{2}\vol(M).
\label{eq.intH4}
\end{align}

Combining \eqref{eq.intH1}, \eqref{eq.intH2}, \eqref{eq.intH3}, and \eqref{eq.intH4} we find
\[
\int_M H \le \bigg(C(M) + \sqrt{\frac{n}{n+1}}\lambda + \eta\bigg) \vol(M).
\]
This completes the proof of Theorem~\ref{thm.main}.

\section{Proof of Theorem~\ref{thm.hypersurface}}

In order to obtain Theorem~\ref{thm.hypersurface}, we need to refine the proof of Theorem~\ref{thm.main}, improving the constant $C(M)$ and making it explicit.
Let $h\in C^\infty(M,\R)$, let $\tau$ be a smooth field of pointwise isometries of $E$, and let $\beta\ge0$ be a constant.

\begin{lemma}
\label{lem.improve}
Assume the situation of the proof of Theorem~\ref{thm.main}.
If $\phi_1,...,\phi_m$ can be chosen such that for each $j$
\[
D_M\phi_j = \varepsilon_j h\phi_j + \beta\tau\phi_j .
\]
holds, where $\varepsilon_j\in \{-1,1\}$, then we have
\begin{align*}
\sum_{j=1}^m \int_M \<D_M\,\Phi_j,\Phi_j\> 
\le
2\,\vol(M)\Bigg(\sqrt{\frac{1}{\vol(M)}\int_M h^2} + \beta \Bigg) .
\end{align*}
\end{lemma}

\begin{proof}
As in the proof of Theorem~\ref{thm.main}, $\phi_1,...,\phi_m$ are scaled such that \eqref{eq.Pphi} holds.
Define the orthogonal projection $P_\tau := \tau^{-1} P^- \tau$.
We compute, for any constants $\alpha_1,\alpha_2>0$:
\begin{align}
\sum_{j=1}^m \int_M \<D_M\,\Phi_j,\Phi_j\> 
&=
2\Re \sum_{j=1}^m \int_M \<D_M\,\phi_j,P^-\Phi_j\> \notag\\
&=
2\Re \sum_{j=1}^m \int_M \<\varepsilon_j h\phi_j + \beta\tau\phi_j,P^-\Phi_j\> \notag\\
&=
2\Re \sum_{j=1}^m \int_M \<\varepsilon_j hP^-\phi_j + \beta P^-\tau\phi_j,P^-\Phi_j\> \notag\\
&=
2\Re \sum_{j=1}^m \int_M \<\varepsilon_j hP^-\phi_j + \beta \tau P_\tau\phi_j,P^-\Phi_j\> \notag\\
&\le
2\sum_{j=1}^m \big(\|hP^-\phi_j\|_{L^2(M)} + \beta\|\tau P_\tau\phi_j\|_{L^2(M)}\big) \|P^-\Phi_j\|_{L^2(M)} \notag\\
&=
2\sum_{j=1}^m \big(\|hP^-\phi_j\|_{L^2(M)} + \beta\|P_\tau\phi_j\|_{L^2(M)}\big) \|P^-\Phi_j\|_{L^2(M)} \notag\\
&\le
2\sum_{j=1}^m \big(\|hP^-\phi_j\|_{L^2(M)} + \beta\|P_\tau\phi_j\|_{L^2(M)}\big) \|P^+\Phi_j\|_{L^2(M)} \notag\\
&=
2\sum_{j=1}^m \big(\|hP^-\phi_j\|_{L^2(M)} + \beta\|P_\tau\phi_j\|_{L^2(M)}\big) \|P^+\phi_j\|_{L^2(M)} \notag\\
&\le
\sum_{j=1}^m \big(\alpha_1\|hP^-\phi_j\|^2 + \tfrac{1}{\alpha_1}\|P^+\phi_j\|^2  + \beta^2 \alpha_2 \|P_\tau\phi_j\|^2 + \tfrac{1}{\alpha_2}\|P^+\phi_j\|^2 \big) \notag\\
&=
\alpha_1 \int_M h^2 + (\tfrac{1}{\alpha_1} + \beta^2\alpha_2 + \tfrac{1}{\alpha_2})\vol(M) .\notag
\label{eq.hbeta}
\end{align}
Choosing $\alpha_1 = \sqrt{\vol(M)/\int_M h^2}$ and $\alpha_2 = 1/\beta$ yields\footnote{If $h$ vanishes identically, then consider the limit $\alpha_1\to\infty$.}
\begin{equation*}
\sum_{j=1}^m \int_M \<D_M\,\Phi_j,\Phi_j\> 
\le
2\sqrt{\vol(M)\int_M h^2} + 2\beta \vol(M) .
\qedhere
\end{equation*}
\end{proof}

Now we prove Theorem~\ref{thm.hypersurface}.
For $\kappa\ge0$, the spinor bundle of $\Mk$ can be trivialized by Killing spinors, i.e., spinors $\Psi_1,...,\Psi_m$ satisfying
\[
\nabla_X \Psi_j = \frac{\sqrt{\kappa}}{2}\gamma(X)\Psi_j.
\]
These Killing spinors are pointwise orthonormal if chosen appropriately.\footnote{For $\kappa<0$, the spinor bundle of $\Mk$ can also be trivialized by Killing spinors (with imaginary constant) but they are not pointwise orthonormal.}

Now let $M\subset\Mk$ be an oriented hypersurface.
We restrict the Killing spinors $\Psi_j$ to $M$ and obtain a global frame $\phi_j := \Psi_j|_M$ of the spinor bundle $E=\Sigma\Mk|_M\to M$.
We rescale this frame as in Section~\ref{sec.CM} so that \eqref{eq.Pphi} holds.
We choose one unit normal field $\nu_0$ for $M$ in $\Mk$ and denote the corresponding mean curvature by $H_0$.
Using \eqref{eq.DiracRel}, we compute
\begin{align*}
D_M \phi_j
&=
\tfrac12 H_0 \phi_j - \gamma(\nu_0) D \Psi_j -\nabla_{{\nu_0}}\Psi_j\\
&=
\tfrac12 H_0 \phi_j + (n+1) \frac{\sqrt{\kappa}}{2} \gamma(\nu_0)\Psi_j - \frac{\sqrt{\kappa}}{2}\gamma(\nu_0)\Psi_j\\
&=
\tfrac12 H_0 \phi_j + \frac{n\sqrt{\kappa}}{2} \gamma(\nu_0)\phi_j. 
\end{align*}
We apply Lemma~\ref{lem.improve} with $h=\frac{H_0}{2}$, $\varepsilon_1=\cdots=\varepsilon_m=1$, $\beta=\frac{n\sqrt{\kappa}}{2}$, and $\tau=\gamma(\nu_0)$ to obtain
\begin{align*}
\sum_{j=1}^m \int_M \<D_M\,\Phi_j,\Phi_j\> 
\le
\vol(M)\Bigg(\sqrt{\frac{1}{\vol(M)}\int_M H_0^2} + n\sqrt{\kappa} \Bigg) 
=
\tfrac12 \, C'(M) \, \vol(M),
\end{align*}
where we have put $C'(M)=2\sqrt{\frac{1}{\vol(M)}\int_M H_0^2} + 2n\sqrt{\kappa}$.
Now the proof of Theorem~\ref{thm.main} carries over with $C(M)$ replaced by $C'(M)$.
This yields Theorem~\ref{thm.hypersurface}.

\section{Proof of Theorem~\ref{thm.eigenvalue}}

For all manifolds listed in the statement of Theorem~\ref{thm.eigenvalue}, the eigenspace of $\tilde{D}_M^2$ for the smallest eigenvalue $\mu_1^2$ has a basis consisting of pointwise orthonormal eigenspinors, see \cite{Fr} for the case of flat tori and \cite{Baer1}*{Section~4} for the spherical cases.
This eigenspace is the orthogonal sum of the eigenspaces of $\tilde{D}_M$ itself for the eigenvalues $\mu_1$ and $-\mu_1$.
Thus, the pointwise orthonormal basis can be chosen to consist of eigenspinors of $\tilde{D}_M$.
This yields a pointwise orthonormal basis $\phi_1,...,\phi_m$ of $E$, consisting of eigensections of $D_M$ for the eigenvalues $\pm\mu_1$ by doubling if $n$ is odd.

We rescale this basis as in Section~\ref{sec.CM} so that \eqref{eq.Pphi} holds.
We apply Lemma~\ref{lem.improve} with $h=\mu_1$ and $\beta=0$ to obtain
\begin{align*}
\sum_{j=1}^m \int_M \<D_M\,\Phi_j,\Phi_j\> 
\le
2\,\vol(M) \sqrt{\frac{1}{\vol(M)}\int_M \mu_1^2}
=
2\,|\mu_1|\,\vol(M) 
=
\tfrac12 C''(M)\,\vol(M),
\end{align*}
where $C''(M)=4|\mu_1|$.
Now the proof of Theorem~\ref{thm.main} carries over with $C(M)$ replaced by $C''(M)$.
This yields Theorem~\ref{thm.eigenvalue}.

\section{Concluding remarks}

The dependence on the lower scalar curvature bound is very explicit and appears to be optimal as the following example shows:

\begin{example}
\label{ex.balls}
Let $X\subset \Mk$ be a compact ball where the radius is chosen such that $M=\partial X$ is the standard unit sphere.
This is possible if $\kappa\le1$.
The mean curvature of $M$ is given by $H = n\sqrt{1-\kappa}$.
For negative $\kappa$ we set $n(n+1)\kappa = -\lambda^2$ with $\lambda > 0$.
Then $\scal_X = -\lambda^2$ and we have
\[
\frac{1}{\vol(M)}\int_M H 
= 
n\cdot \sqrt{1+ \tfrac{\lambda^2}{n(n+1)}}
=
\sqrt{\tfrac{n^2}{\lambda^2}+ \tfrac{n}{n+1}}\cdot\lambda 
\sim
\sqrt{\tfrac{n}{n+1}}\cdot\lambda \quad\text{ as }\lambda\to\infty.
\]
This shows that the dependence on $\lambda$ in our estimate is of the right order for large $\lambda$.
\end{example}

It seems strange that one needs a lower bound on the mean curvature in order to obtain an upper bound on its integral.
It is unclear to the author whether this dependence can be removed.
Note that the Gauss-Bonnet argument for $\dim(M)=1$ and $\lambda=0$ in the introduction does not need a lower bound on $H$.

\begin{remark}
The definition of the constant $C(M)$ in Section~\ref{sec.CM} depends on the spin structure of $M$.
However, since $M$ is compact, it has only finitely many spin structures.
Thus, $C(M)$ can be chosen independently of the spin structure.
\end{remark}


\end{document}